\documentclass[a4paper, 11pt]{article}

\usepackage[utf8]{inputenc}
\usepackage[T1]{fontenc}
\usepackage[a4paper, margin=2cm]{geometry}
\usepackage{needspace}

\usepackage{libertine} 
\usepackage{inconsolata} 

\usepackage{amsmath, amsthm, amssymb}
\usepackage{graphicx}
\usepackage{enumerate}
\usepackage{authblk}
\usepackage[textsize=small, textwidth=2cm, color=yellow]{todonotes}
\usepackage{thmtools}
\usepackage{thm-restate}
\usepackage[colorlinks=true, citecolor=red]{hyperref}
\usepackage{float}

\declaretheorem[name=Theorem]{theorem}
\declaretheorem[name=Lemma, sibling=theorem]{lemma}
\declaretheorem[name=Proposition, sibling=theorem]{proposition}

\declaretheorem[name=Conjecture, sibling=theorem]{conjecture}

\def\cqedsymbol{\ifmmode$\lrcorner$\else{\unskip\nobreak\hfil
\penalty50\hskip1em\null\nobreak\hfil$\lrcorner$
\parfillskip=0pt\finalhyphendemerits=0\endgraf}\fi}

\let\le\leqslant
\let\ge\geqslant

\let\geq\geqslant

\title{Chromatic number of spacetime}

\author[1]{James Davies}

\affil[1]{University of Cambridge, United Kingdom.}
\date{}

\begin{document}

\maketitle

\begin{abstract}
	We observe that an old theorem of Graham implies that for any positive integer $s$, there exists some positive integer $T(s)$ such that every $s$-colouring of $\mathbb{Z}^2$ contains a monochromatic pair of points $(x,y),(x',y')$ with $
	(x-x')^2 - (y-y')^2 = (T(s))^2$.
	By scaling, this implies that every finite colouring of $\mathbb{Q}^2$ contains a monochromatic pair of points $(x,y),(x',y')$ with $
	(x-x')^2 - (y-y')^2 = 1$, which answers in a strong sense a problem of Kosheleva and Kreinovich on a pseudo-Euclidean analogue of the Hadwiger-Nelson problem.
	
	The proof of Graham's theorem relies on repeated applications of van der Waerden's theorem, and so the resulting function $T(s)$ grows extremely quickly.
    We give an alternative proof in the weaker setting of having a second spacial dimension that results in a significantly improved bound.
    To be more precise, we prove that for every positive integer $s$ with $r\equiv 2 \pmod{4}$, every $s$-colouring of $\mathbb{Z}^3$ contains a monochromatic pair of points $(x,y,z),(x',y',z')$ such that $
    (x-x')^2 + (y-y')^2 - (z-z')^2 =  (5^{(s-2)/4}(8\cdot 5^{(s-2)/2})!)^2$.
    In fact, we prove a stronger density version.
    The density version in $\mathbb{Z}^2$ remains open.
\end{abstract}

\section{Introduction}\label{intro}

The Hadwiger-Nelson problem asks for the minimum number of colours required to colour $\mathbb{R}^2$ so as to avoid any monochromatic pair of points at unit distance apart. The upper bound of 7 was observed by Isbell in 1950 (see~\cite{soifer2009mathematical}), and the lower bound of 5 is a recent breakthrough of de Grey~\cite{de2018chromatic} and independently Exoo and Ismailescu~\cite{exoo2020chromatic}.
For more on the history of the Hadwiger-Nelson problem and many of its relatives, see~\cite{soifer2009mathematical}.

In 2009, Kosheleva and Kreinovich~\cite{kosheleva2008chromatic}, proposed the study of the analogue of the Hadwiger-Nelson problem for pseudo-Euclidean spaces.
For positive integers $p,q$, they define $\chi(\mathbb{R}^{p,q})$ to be equal to the minimum number of colours required to colour $\mathbb{R}^{p+q}$ so that there is no monochromatic pair of points $(a_1, \ldots , a_{p+q}), (b_1, \ldots , b_{p+q})$ with $\sum_{j=1}^p (a_j -b_j)^2 - \sum_{j=p+1}^{p+ q} (a_j-b_j)^2 = 1$.
They asked to determine the value of $\chi(\mathbb{R}^{p, q})$, and in particular, whether or not $\chi(\mathbb{R}^{p, q})$ is finite.
Note that if $p\le p'$, $q\le q'$, then $\chi(\mathbb{R}^{p, q}) \le \chi(\mathbb{R}^{p', q'})$ since we can restrict some coordinates to $0$.
We observe that one can use an old theorem of Graham \cite{graham1980} to show that $\chi(\mathbb{R}^{1, 1})= \infty$. Thus, $\chi(\mathbb{R}^{p, q})= \infty$ for all pairs $p,q\ge 1$.
Graham \cite{graham1980} proved that every finite colouring of the plane contains a monochromatic triple of points forming a triangle of unit area. We require a strengthening proven by Graham~\cite{graham1980}.

\begin{theorem}[Graham {\cite[Theorem~1]{graham1980}}]\label{Graham}
	For any positive integer $s$, there exists a positive integer $T(s)$ so that in any
	$s$-colouring of $\mathbb{Z}^2$, there is always a monochromatic
	triple $(u,v),(u+a,v),(u,v+b)\in \mathbb{Z}^2$ with $ab=T(s)$. 
\end{theorem}

\begin{theorem}\label{T}
	For any positive integer $s$, there exists a positive integer $T(s)$ so that in any
	$s$-colouring of $\mathbb{Z}^2$, there is always a monochromatic
	pair $(x,y),(x',y')\in \mathbb{Z}^2$ with $(x-x')^2 - (y-y')^2 = (2T(s))^2$.
\end{theorem}

\begin{proof}
	Consider some $s$-colouring of $\mathbb{Z}^2$.
	By Theorem \ref{Graham}, there exists a monochromatic triple of points $(u,v),(u,v)+a(1,1),(u,v)+b(-T(s),T(s)) \in \mathbb{Z}[(1,1),(-T(s),T(s))]\subset \mathbb{Z}^2$ with $ab=T(s)$.
	Then, observe that $(a+bT(s))^2 - (a-bT(s))^2 = 4abT(s) = (2T(s))^2$.
	Therefore, $(x,y)= (u,v)+a(1,1)$, $(x',y') = (u,v)+b(-T(s),T(s))$ is a monochromatic pair of points in $\mathbb{Z}^2$ with $(x-x')^2 - (y-y')^2 = (2T(s))^2$, as desired.
\end{proof}

By scaling, Theorem \ref{T} then implies the following, which resolves Kosheleva and Kreinovich's~\cite{kosheleva2008chromatic} problem.

\begin{theorem}\label{main1}
	Every finite colouring of $\mathbb{Q}^2$ contains a monochromatic pair of points $(x,y),(x',y')$ with $
	{(x-x')^2} - (y-y')^2 = 1$.
	In particular, $\chi(\mathbb{R}^{1, 1})= \infty$.
\end{theorem}

For every positive integer $d$, there is a finite colouring of $\mathbb{R}^d$ avoiding monochromatic pairs of points at unit distance apart~\cite{larman1972realization}. So, Theorem~\ref{main1} highlights a subtle but significant difference between the geometry of Euclidean spaces and certain pseudo-Euclidean spaces.
This also answers a question of Geelen~\cite{Jim-personal}, who independently asked whether or not $\chi(\mathbb{R}^{2, 2})$ is finite.
Kosheleva and Kreinovich's~\cite{kosheleva2008chromatic} problem was also independently raised by Madore \cite{Madore} who further observed that $\chi(\mathbb{R}^{1, 1})= \infty$ implies that $\chi(\mathbb{C}^{2})= \infty$, where $\chi(\mathbb{C}^{2})$ is equal to the minimum number of colours required to colour $\mathbb{C}^2$ so that there is no monochromatic pair of points $(x,y),(x',y')\in \mathbb{C}^2$ with $(x-x')^2 + (y-y')^2=1$.

Graham's \cite{graham1980} proof of Theorem \ref{Graham} relies on repeated applications of van der Waerden's theorem \cite{van}, and thus the value of $T(s)$ in Theorem \ref{Graham} and Theorem \ref{T} grows extremely fast.
We remark that ``$(2T(s))^2$'' cannot be replaced with ``$1$'' in Theorem \ref{T}, since one can avoid such monochromatic pairs in $\mathbb{Z}^2$ by 2-colouring the points $(x,y) \in \mathbb{Z}^2$ according to whether $x+y$ is odd or even.
We give an alternative proof of a weakening of Theorem \ref{T} where there is a second spacial dimension. Although this theorem is in a weaker setting, we obtain significantly improved bounds.

\begin{theorem}\label{main3}
	Let $s$ be a positive integer with $s\equiv 2 \pmod{4}$.
	Then, every $s$-colouring of $\mathbb{Z}^3$ contains a monochromatic pair of points $(x,y,z),(x',y',z')$ with $
	(x-x')^2 + (y-y')^2 - (z-z')^2 =  (5^{(s-2)/4}(8\cdot 5^{(s-2)/2})!)^2$.
\end{theorem}

Our proof of Theorem~\ref{main3} uses the method recently introduced by the author~\cite{davies2022odd} to show that every finite colouring of $\mathbb{R}^2$ contains a monochromatic pair of points whose distance is an odd integer.
With McCarty and Pilipczuk~\cite{davies2023prime}, we also very recently extended these methods to prove that every finite colouring of $\mathbb{R}^2$ contains monochromatic prime and polynomial distances.
This method is applied to subgraphs that are Cayley graphs of $\mathbb{Z}^d$.
Given a set $I\subseteq \mathbb{Z}^d$, its \emph{upper density} is
\[
\delta(I) = \limsup_{R \to \infty} \frac{|I \cap [-R,R]^d  |}{(2R+1)^d}.
\]
For a positive integer $r$, we let $
D_r= \{ (x,y) \in \mathbb{Z}^2 : x^2 + y^2 = r^2 \}$.
In other words, $D_r$ is the set of points contained in both the integer lattice $\mathbb{Z}^2$ and the circle of radius $r$ and center $(0,0)$.
We prove the following.

\begin{theorem}\label{main2}
    Let $r$ be a positive integer.
    If $I\subseteq \mathbb{Z}^3$ has upper density greater than $\frac{1}{1+|D_r|/2}$, then there is a pair of points $(x,y,z),(x',y',z')\in I$ with $
    (x-x')^2 + (y-y')^2 - (z-z')^2 = (r(8r^2)!)^2$.
    
    As a consequence, every $(|D_r|/2)$-colouring of $\mathbb{Z}^3$ contains a monochromatic pair of points $(x,y,z),(x',y',z')$ with $
    (x-x')^2 + (y-y')^2 - (z-z')^2 = (r(8r^2)!)^2$.
\end{theorem}

If $r$ has prime factorization $r=2^a \prod p_i^{b_i} \prod q_i^{c_i}$ where $p_i \equiv 1 \pmod{4}$ and $q_i \equiv 3 \pmod{4}$, then by Jacobi's two-square theorem (see \cite{hardybook}), we have that $|D_r|=4 \prod (2b_i +1)$.
Therefore, by Theorem \ref{main2} with $r=5^b$, every $(4b+2)$-colouring of $\mathbb{Z}^3$ contains a monochromatic pair of points $(x,y,z),(x',y',z')$ with $
(x-x')^2 + (y-y')^2 - (z-z')^2 = (5^{b}(8\cdot 5^{2b})!)^2$. Theorem \ref{main3} then follows from this by setting $b=(s-2)/4$. We remark that one could of course obtain slightly stronger bounds in Theorem \ref{main3}, but we choose to provide a simple explicit bound.

Note that Theorem \ref{main2} further provides a density strengthening of Theorem \ref{main3}.
We conjecture an analogues density strengthening of Theorem \ref{T}.

\begin{conjecture}\label{conjecture}
    For every $\epsilon > 0$, there exists some positive integer $d$ so that if $I\subseteq \mathbb{Z}^2$ has upper density greater than $\epsilon$, then there is a pair of points $(x,y),(x',y')\in I$ with $
    (x-x')^2 - (y-y')^2 = d^2$.
\end{conjecture}

Of course, a density version of Graham's theorem \cite{graham1980} (Theorem \ref{Graham}) would imply Conjecture \ref{conjecture}, but this remains open (see \cite{Kovac}). Bardestani and Mallahi-Karai \cite{Bardestani} proved a measurable analogue of Conjecture \ref{conjecture} in $\mathbb{R}^2$ (see also \cite{Kovac}).
Kovač \cite{Kovac} very recently proved a density version of Graham's theorem \cite{graham1980} in the measurable setting.
The following weakening of Conjecture \ref{conjecture} with more temporal dimensions also remains open.

\begin{conjecture}\label{conjecture2}
    There exists an integer $m\ge 2$ with the following property.
    For every $\epsilon > 0$, there exists some positive integer $d$ so that if $I\subseteq \mathbb{Z}^m$ has upper density greater than $\epsilon$, then there is a pair of points $(x_1,\ldots , x_m),(x_1', \ldots x_m')\in I$ with $
    (x_1-x_1')^2 - \sum_{j=2}^m (x_j-x_j')^2 = d^2$.
\end{conjecture}

In Section~\ref{prelims} we introduce some preliminaries for proving Theorem~\ref{main2}. This includes an analogue of the Lov{\'a}sz theta bound~\cite{lovasz1979shannon} for Cayley graphs of $\mathbb{Z}^d$~\cite{davies2022odd}, which shall be our main tool.
Then, in Section~\ref{proof}, we prove Theorem~\ref{main2} (and thus Theorem~\ref{main3}).

\section{Preliminaries}\label{prelims}

The \emph{chromatic number} $\chi(G)$ of a graph $G$ is equal to the minimum number of colours required to assign each vertex a colour so that no two adjacent vertices receive the same colour. An \emph{independent set} in a graph $G$ is a set of pairwise non-adjacent vertices.

We say that a set $C\subseteq \mathbb{Z}^d$ is \emph{centrally symmetric} if $C=-C$. Similarly, a function $w: C \to \mathbb{R}$ is centrally symmetric if $w(x)=w(-x)$ for all $x\in C$. For a centrally symmetric set $C\subseteq \mathbb{Z}^d \backslash \{0\}$, we let $G(\mathbb{Z}^d, C)$ be the Cayley graph of $\mathbb{Z}^d$ with generating set $C$.
In other words, $G(\mathbb{Z}^d, C)$ is the graph with vertex set $\mathbb{Z}^d$ where two vertices $u,v\in \mathbb{Z}^d$ are adjacent if $u-v\in C$.
For a Cayley graph $G(\mathbb{Z}^d , C)$, we let
\[
\overline{\alpha}(G(\mathbb{Z}^d, C)) = \sup \{ {\delta(I)} : I \text{ is an independent set of } G(\mathbb{Z}^d, C) \}.
\]
Notice that $\chi(G(\mathbb{Z}^d, C)) \ge 1/ \overline{\alpha}(G(\mathbb{Z}^d, C))$.
To see this inequality, note that if $V(G)$ can be partitioned into $k$ independent sets $I_1, \ldots, I_k$, then $k\overline{\alpha}(G) \geq \delta(I_1)+\dots+\delta(I_k) \geq \delta(V(G))=1$.

For $\alpha \in \mathbb{R}$, we let $e(\alpha) = e^{2 \pi i \alpha}$.
Since $e(\alpha)=e(\alpha +1)$, for $a\in \mathbb{R}/\mathbb{Z}$ we can also define $e(a)$ to be equal to $e(b)$, where $0\le b <1$ is such that $b\equiv a \pmod{1}$.
For $x, y \in \mathbb{Z}^d$, we write $x \cdot y$ for the dot product.

The following is a slightly simplified version of the ratio bound given in~\cite[Theorem~3]{davies2022odd}, which was used to show that the odd distance graph has unbounded chromatic number. This is our main tool and is essentially an analogue of the Lov{\'a}sz theta bound~\cite{lovasz1979shannon} for Cayley graphs of $\mathbb{Z}^d$.

\begin{theorem}[Davies {\cite[Theorem~3]{davies2022odd}}]\label{ratio bound}
    Let $C\subseteq \mathbb{Z}^d \backslash \{0\}$ and $w:C \to \mathbb{R}_{\ge 0}$ be centrally symmetric with $\sum_{x\in C} w(x)$ positive. Then
	\[
	\overline{\alpha}(G(\mathbb{Z}^d, C)) \le \frac{-\inf_{u\in (\mathbb{R}/\mathbb{Z})^d}  \widehat{w}(u) }{  \sup_{u\in (\mathbb{R}/\mathbb{Z})^d}  \widehat{w}(u)   -\inf_{u\in (\mathbb{R}/\mathbb{Z})^d}  \widehat{w}(u)},
	\]
	where
	\[
	\widehat{w}(u) = \sum_{x\in C} w(x) e( u \cdot x).
	\]
\end{theorem}

Note that $\inf_{u\in (\mathbb{R}/\mathbb{Z})^d}  \widehat{w}(u) = \inf_{u\in \mathbb{R}^d}  \widehat{w}(u)$, and $\sup_{u\in (\mathbb{R}/\mathbb{Z})^d}  \widehat{w}(u) = \sup_{u\in \mathbb{R}^d}  \widehat{w}(u)$.

To apply Theorem~\ref{ratio bound} effectively, we need good estimates for the exponential sums that shall appear.
The supremum will be straightforward to estimate, while the infimum is trickier. The rest of this section is dedicated to proving three lemmas that will be used to estimate the infimum (which is done in Lemma~\ref{inf}).

The first two straightforward lemmas shall allow us to estimate two simpler exponential sums that will appear in our proof.
Similar statements appear in~\cite{davies2022odd}.

\begin{lemma}\label{zero}
    Let $(a_n)_{n=1}^\infty$, $(b_n)_{n=1}^\infty$ be sequences in $\mathbb{R}/ \mathbb{Z}$ such that $\lim_{n \to \infty} a_n \equiv \lim_{n \to \infty} b_n \equiv 0 \pmod{1}$.
    Then,
    \[
    \liminf_{n \to \infty} \frac{1}{n } \sum_{j=1}^{n} \frac{n + 1 -j}{ n+1 } ( e( a_n + jb_n  ) + e( -a_n - jb_n  ))
		\ge 0.
    \]
\end{lemma}

\begin{proof}
    By the identity $e(\theta)+e(-\theta)= 2 \cos(2\pi \theta)$, it is enough to show that
    \[
    \liminf_{n \to \infty} \frac{1}{n } \sum_{j=1}^{n} \frac{n +1 -j}{ n+1 } \cos(2\pi (a_n + j b_n ))
		\ge 0.
    \]
    Since $\lim_{n \to \infty} a_n \equiv 0 \pmod{1}$, we have that
    \[
    \liminf_{n \to \infty} \frac{1}{n } \sum_{j=1}^{n} \frac{n +1 -j}{ n+1 } \cos(2\pi (a_n + j b_n ))
		=
    \liminf_{n \to \infty} \frac{1}{n } \sum_{j=1}^{n} \frac{n +1 -j}{n+1 } \cos(2\pi j b_n ).
    \]
    For each $n\ge 1$, let $0\le \theta_n <1$ be such that $\theta_n \equiv b_n \pmod{1}$.
    If $\theta_n=0$, then observe that
    \[
    \frac{1}{n } \sum_{j=1}^{n} \frac{n +1 -j}{n+1 } \cos(2\pi  j \theta_n )
    =
    \frac{1}{n } \sum_{j=1}^{n} \frac{n +1 -j}{n+1 } 
    = \frac{1}{2}
    > 0.
    \]
    So, we may assume that $\theta_n \not= 0$.

    Observe that
    \[
    \lim_{n \to \infty} \left[ 
    \frac{1}{n } \sum_{j=1}^{n} \frac{n +1 -j}{n+1 } \cos(2\pi j  \theta_n )
    -
    \frac{1}{ n \theta_{n} } \int_{0}^{ n\theta_{n}  }
		\left( 1 - \frac{t}{  n \theta_{n} } \right) \cos(2 \pi  t) \,dt
    \right]
    =0.
    \]
    Therefore, 
     \[
    \liminf_{n \to \infty} \frac{1}{n } \sum_{j=1}^{n} \frac{n + 1 -j}{ n+1 } ( e( a_n + j b_n  ) + e( -a_n - j b_n  ))
    =
    \frac{1}{ n \theta_{n} } \int_{0}^{ n\theta_{n}  }
		\left( 1 - \frac{t}{  n \theta_{n} } \right) \cos(2 \pi  t) \,dt .
    \]
    So, let us evaluate the above integral.
    
    \begin{align*}
		\frac{1}{  n \theta_n  }   \int_{0}^{ n \theta_n  }
		\left( 1 - \frac{ t}{  n \theta_n  } \right) \cos(2\pi t) \,dt 
		& =\frac{1}{ 2 \pi n \theta_n  }   \int_{0}^{ 2\pi n \theta_n  }
		\left( 1 - \frac{t}{ 2\pi n \theta_n  } \right) \cos( t) \,dt \\
		& = \frac{1}{ 2\pi  n \theta_n  }   \int_{0}^{ 2\pi n \theta_n   } \cos(t) \,dt
		- \frac{1}{ 4 \pi^2 n^2 \theta_n^2 }   \int_{0}^{ 2\pi n \theta_n   }
		t \cos(t) \,dt\\
		& = \frac{ \sin( 2\pi  n \theta_n  )  }{ 2\pi n \theta_n  }
		- \frac{ \sin( 2\pi n \theta_n  )  }{ 2\pi n \theta_n  }
		+ \frac{1}{ 4\pi^2 n^2 \theta_n^2  } \int_{0}^{ 
2\pi n \theta_n   } \sin(t) \,dt\\
		& = \frac{1 - \cos( 2\pi n \theta_n )  }{ 4 \pi^2  n^2 \theta_n^2  }.
		\end{align*}
		The lemma now follows since $\frac{1 - \cos( 2\pi n \theta_n )  }{ 4 \pi^2  n^2 \theta_n^2  } \ge 0$ for all $\theta_n \not=0$.
\end{proof}

\begin{lemma}\label{geometric series}
    Let $(a_n)_{n=1}^\infty$, $(b_n)_{n=1}^\infty$ be sequences in $\mathbb{R}/ \mathbb{Z}$ such that $\lim_{n \to \infty} a_n \equiv a \pmod{1}$, and $\lim_{n \to \infty} b_n \equiv b \not\equiv 0 \pmod{1}$.
    Then,
    \[
    \lim_{n \to \infty}  \frac{1}{n} \sum_{j=1}^{n} \frac{n +1 -j}{ n+1 } e( a_n + jb_n  ) 
		= 0.
    \]
\end{lemma}

\begin{proof}
Observe that
\begin{align*}
	\left| \frac{1}{n} \sum_{j=1}^{n} \frac{n +1 -j}{ n+1 } e( j b_n  ) \right|
    & = 
    \frac{1}{n} \left|
    \sum_{m=0}^{\lfloor \sqrt{n} \rfloor -1}
    \sum_{j=m\lfloor \sqrt{n} \rfloor +1}^{(m+1) \lfloor \sqrt{n} \rfloor} \frac{n +1 -j}{ n+1 } e( j b_n  )
    + \sum_{j= \lfloor \sqrt{n} \rfloor^2 + 1}^n \frac{n +1 -j}{ n+1 } e( j b_n  )
    \right|
    \\
    & \le
    \frac{1}{n} \left|
    \sum_{m=0}^{\lfloor \sqrt{n} \rfloor -1}
    \sum_{j=m\lfloor \sqrt{n} \rfloor +1}^{(m+1) \lfloor \sqrt{n} \rfloor} \frac{n +1 -j}{ n+1 } e( j b_n  ) \right|
    + \frac{1}{\sqrt{n}}
    \\
    & \le
    \frac{1}{n} \left|
    \sum_{m=0}^{\lfloor \sqrt{n} \rfloor -1}
    \frac{n -(m+1)\lfloor \sqrt{n} \rfloor}{ n+1 }
    \sum_{j=m\lfloor \sqrt{n} \rfloor +1}^{(m+1) \lfloor \sqrt{n} \rfloor}  e(  j b_n  ) \right|
    +  \lfloor \sqrt{n} \rfloor  \frac{ \lfloor \sqrt{n}  \rfloor   }{n(n+1)}
    + \frac{1}{\sqrt{n}}
    \\
    & \le
    \frac{1}{n} 
    \sum_{m=0}^{\lfloor \sqrt{n} \rfloor -1}
    \frac{n -(m+1)\lfloor \sqrt{n} \rfloor}{ n+1 }
    \left|
    \sum_{j=m\lfloor \sqrt{n} \rfloor +1}^{(m+1) \lfloor \sqrt{n} \rfloor}  e(  j b_n  ) \right|
    + \frac{2}{\sqrt{n}}
    \\ 
    & \le
    \frac{1}{n} 
    \sum_{m=0}^{\lfloor \sqrt{n} \rfloor -1}
    \left|
    \sum_{j=0}^{\lfloor \sqrt{n} \rfloor -1}  e(  j b_n  ) \right|
    + \frac{2}{\sqrt{n}}
    \\
    & =
    \frac{1}{n} 
    \lfloor \sqrt{n} \rfloor
    \left|
    \frac{1-e( \lfloor \sqrt{n} \rfloor b_n )}{1-e(b_n)}
    \right|
    + \frac{2}{\sqrt{n}}
    \\
    & \le
    \frac{2}{\sqrt{n}|1-e(b_n)|} 
    + \frac{2}{\sqrt{n}}
    \\
    & =
    \frac{2}{\sqrt{n}} \left( \frac{1}{|1-e(b_n)|} + 1 \right).
\end{align*}
Since $\lim_{n \to \infty} b_n \equiv b \not\equiv 0 \pmod{1}$, we have that $|1-e(b_n)|$ is eventually bounded away from $0$. So then, $\frac{1}{|1-e(b_n)|} + 1$ is bounded.
Therefore,

\begin{align*}
	\lim_{n \to \infty} \left| \frac{1}{n} \sum_{j=1}^{n} \frac{n +1 -j}{ n+1 } e( a_n + j b_n  ) \right|
    & =
    \lim_{n \to \infty} \left| \frac{1}{n} \sum_{j=1}^{n} \frac{n +1 -j}{ n+1 } e( j b_n  ) \right|
    \\
    & \le
    \lim_{n \to \infty}
    \frac{2}{\sqrt{n}} \left( \frac{1}{|1-e(b_n)|} + 1 \right) = 0,
\end{align*}
as desired.
\end{proof}

The exponential sums that we must examine to prove Theorem~\ref{main2} are more complicated than those given Lemma~\ref{zero} and Lemma~\ref{geometric series}.
However, the following key lemma will allow us to reduce the examination of these more complicated exponential to those given in Lemma~\ref{zero} and Lemma~\ref{geometric series}.

\begin{lemma}\label{circle}
    Let $r$ be a positive integer, let $(x_1,y_1), (x_2,y_2), (x_3,y_3) \in D_r$ be distinct, and let $a,b\in \mathbb{R}$ be such that $(x_1 , y_1) \cdot (a,b) \equiv (x_2 , y_2) \cdot (a,b) \equiv (x_3 , y_3) \cdot (a,b) \pmod{1}$.
    Then, there exists integers $p_a,q_a,p_b,q_b$ such that $a=\frac{p_a}{q_a}$, $b=\frac{p_b}{q_b}$, and $1\le q_a,q_b \le 8r^2$.
\end{lemma}

\begin{proof}
    We will argue that there exists such integers $p_a,q_a$ since the existence of such integers $p_b,q_b$ follows similarly.

    Since $(x_1 , y_1) \cdot (a,b) \equiv (x_2 , y_2) \cdot (a,b) \equiv (x_3 , y_3) \cdot (a,b) \pmod{1}$, we have that 
    $\left( (x_2,y_2) - (x_1,y_1) \right) \cdot (a,b) \equiv 0 \pmod{1}$, and 
    $\left( (x_3,y_3) - (x_1,y_1) \right) \cdot (a,b) \equiv 0 \pmod{1}$.
    So, both $\left( (x_2,y_2) - (x_1,y_1) \right) \cdot (a,b) = a(x_2-x_1) + b(y_2-y_1)$ and $\left( (x_3,y_3) - (x_1,y_1) \right) \cdot (a,b) = a(x_3-x_1) + b(y_3-y_1)$ are integers.
    Therefore, $
    (y_3-y_1) \left( a(x_2-x_1) + b(y_2-y_1) \right) - (y_2-y_1) \left( a(x_3-x_1) + b(y_3-y_1) \right)
    =
    a ((x_2-x_1)(y_3-y_1) - (x_3-x_1)(y_2-y_1) )$ is also an integer.

    Since $(x_1,y_1), (x_2,y_2), (x_3,y_3)$ are all distinct, and all lie on a circle, we have that $(x_2-x_1, y_2-y_1)$ and $(x_3-x_1, y_3-y_1)$ are independent in $\mathbb{Z}^2$.
    Additionally, $y_1,y_2,y_3$ are not all equal, so at most one of $y_3-y_1$ and $y_2-y_1$ is zero.
    Therefore, $(y_3-y_1)(x_2-x_1, y_2-y_1) - (y_2-y_1)(x_3-x_1, y_3-y_1) = (    (x_2-x_1)(y_3-y_1) - (x_3-x_1)(y_2-y_1)         , 0)$ is non-zero.
    So, $(x_2-x_1)(y_3-y_1) - (x_3-x_1)(y_2-y_1)$ is non-zero.
    Note that, since $(x_1,y_1), (x_2,y_2), (x_3,y_3) \in D_r$, we have that $|x_1|,|x_2|,|x_3|,|y_1|,|y_2|,|y_3| \le r$.
    So, we further have that \[|(x_2-x_1)(y_3-y_1) - (x_3-x_1)(y_2-y_1)| \le (|x_2|+|x_1|)(|y_3|+|y_1|) + (|x_3|+|x_1|)(|y_2|+|y_1|) \le 8r^2.\]

    As $a ((x_2-x_1)(y_3-y_1) - (x_3-x_1)(y_2-y_1) )$ is an integer, and $(x_2-x_1)(y_3-y_1) - (x_3-x_1)(y_2-y_1)$ is a non-zero integer, it now follows that there exists integers $p_a,q_a$ such that $a=\frac{p_a}{q_a}$, and $1\le q_a \le 8r^2$. We can argue that there exists integers $p_b,q_b$ such that $b=\frac{p_b}{q_b}$, and $1\le q_b \le 8r^2$ similarly.
\end{proof}

\section{Proof}\label{proof}

In this section we prove Theorem~\ref{main2}, and thus Theorem~\ref{main3}.

For $d\in D_r$, we let $\tilde{d}$ be the element of $D_r$, such that the clockwise angle from $\tilde{d}$ to $d$ is equal to $\pi / 2$. In other words, if $d=(x,y)$, then $\tilde{d}=(-y,x)$. For each positive integer $r$, we define the following subset of $\mathbb{Z}^3\backslash \{0\}$, which shall be the generating set of our Cayley graph;
\[
C_{r} = \{  \pm ( (8r^2)!d + j \tilde{d}, jr ) : d\in D_r,  j\in \mathbb{N} \}.
\]

Now we show that $G(\mathbb{Z}^3, C_r)$ is a suitable distance graph in $\mathbb{Z}^3$.

\begin{lemma}\label{embed}
    For every positive integer $r$, $G(\mathbb{Z}^3, C_r)$ is a subgraph of the graph on the same vertex set $\mathbb{Z}^3$, where $(x,y,z),(x',y',z')\in \mathbb{Z}^3$ are adjacent if $(x-x')^2+(y-y')^2-(z-z')^2= (r(8r^2)!)^2$.
\end{lemma}

\begin{proof}
    By transitivity, it is enough to show that if $(x,y,z)\in C_r$, then $x^2 + y^2 - z^2 = (r(8r^2)!)^2$.
    Observe that for each $(x,y,z)\in C_r$, there exists a triple $\sigma\in \{1,-1\}$, $d=(d_1,d_2)\in D_r$, $j\in \mathbb{N}$ such that $(x,y,z) = \sigma( (8r^2)!d + j \tilde{d}, jr )
    = ( \sigma ( (8r^2)!d_1-jd_2 ) , \sigma ((8r^2)!d_2 + jd_1 ) , \sigma jr  )$.
    Since $d=(d_1,d_2)$ and $\tilde{d}=(-d_2,d_1)$ are orthogonal, we have that
    \begin{align*}
        ( \sigma ( (8r^2)!d_1-jd_2))^2 +  ( \sigma ( (8r^2)!d_2 +  jd_1))^2 - (\sigma jr  )^2
    &=
    \| (8r^2)!d + j \tilde{d}  \|_2^2 - j^2r^2
    \\
    &=
     \| (8r^2)!d \|_2^2 + \| j \tilde{d}  \|_2^2 - j^2r^2
     \\
    &=
     ((8r^2)!)^2\| d \|_2^2 + j^2\|  \tilde{d}  \|_2^2 - j^2r^2
     \\
    &=
    ((8r^2)!)^2 r^2 + j^2r^2 - j^2r^2
     \\
    &=
    (r(8r^2)!)^2,
    \end{align*}
    as desired.
\end{proof}

To apply Theorem~\ref{ratio bound}, it is convenient to consider finite subsets of $C_r$. For each pair of positive integers $r,n$, let
\[
C_{r,n} = \{  \pm ((8r^2)!d + j \tilde{d}, jr) : d\in D_r, 1\le j \le n \},
\]
and let $w_{r,n} : C_{r,n} \to \mathbb{R}_{\ge 0}$ be such that for each $c\in C_{r,n}$ of the form $c= \pm ((8r^2)!d + j \tilde{d}, jr)$, we have $w_{r,n}(c) = \frac{n+1-j}{n(n+1)}$.
Then as in Theorem~\ref{ratio bound}, for $u\in (\mathbb{R}/ \mathbb{Z} )^3$ we have that
\[
\widehat{w}_{r,n}(u)
=
\sum_{d\in D_r} \sum_{\sigma= \pm 1} \sum_{j=1}^n \frac{n+1-j}{n(n+1)} e( \sigma ((8r^2)!d + j \tilde{d}, jr) \cdot u).
\]

To obtain good bounds from Theorem~\ref{ratio bound}, we need good estimates for $ \sup_{u\in (\mathbb{R}/\mathbb{Z})^3} \widehat{w}_{r,n}(u) $ and $\inf_{u\in (\mathbb{R}/\mathbb{Z})^3} \widehat{w}_{r,n}(u)$.
We begin by evaluating $ \sup_{u\in (\mathbb{R}/\mathbb{Z})^3} \widehat{w}_{r,n}(u) $.

\begin{lemma}\label{sup}
    For positive integers $r,n$, we have that
    $
    \sup_{u\in (\mathbb{R}/\mathbb{Z})^3} \widehat{w}_{r,n}(u) = |D_r|
    $.
\end{lemma}

\begin{proof}
    Clearly $\sup_{u\in (\mathbb{R}/\mathbb{Z})^3} \widehat{w}_{r,n}(u) = 
     \widehat{w}_{r,n}(0)
     = 2|D_r|\sum_{j=1}^n \frac{n+1-j}{n(n+1)}
     =|D_r|$.
\end{proof}

Now, we bound $\limsup_{n\to \infty} \inf_{u\in (\mathbb{R}/\mathbb{Z})^3} \widehat{w}_{r,n}(u)$ using the three lemmas from Section~\ref{prelims}.

\begin{lemma}\label{inf}
    For each positive integer $r$, we have that
    $
    \limsup_{n\to \infty} \inf_{u\in (\mathbb{R}/\mathbb{Z})^3} \widehat{w}_{r,n}(u) \ge -2
    $.
\end{lemma}

\begin{proof}
    Since $(\mathbb{R}/\mathbb{Z})^3$ is compact, for each positive integer $n$, there exists some $(a_n,b_n,c_n)\in (\mathbb{R}/\mathbb{Z})^3$ such that $\inf_{u\in (\mathbb{R}/\mathbb{Z})^3} \widehat{w}_{r,n}(u) = \widehat{w}_{r,n}(a_n,b_n,c_n)$.
    By the Bolzano-Weierstrass theorem, there is a strictly
    increasing sequence of positive integers $(n_f)_{f=1}^\infty$ such that the sequence $\left( (a_{n_f} , b_{n_f}, c_{n_f})  \right)_{f=1}^\infty$ converges to some point $(a,b,c)\in (\mathbb{R}/\mathbb{Z})^3$. Fix such $a,b,c\in \mathbb{R}$.
    It is enough now to show that 
    $
    \limsup_{f\to \infty} \widehat{w}_{r,n}(a_{n_f} , b_{n_f}, c_{n_f})   \ge -2
    $.
    Let $D_r'$ be the elements $d$ of $D_r$ such that $(\tilde{d} , r) \cdot (a,b,c)\equiv 0 \pmod{1}$.
    Let $D_r^*=D_r\backslash D_r'$.

    For $d\in D_r^*$ and $1\le j \le n_f$,
    observe that
    \[
    ((8r^2)!d + j \tilde{d}, jr) \cdot (a_{n_f},b_{n_f},c_{n_f})
    =
      (8r^2)!d \cdot (a_{n_f},b_{n_f}) + j ( \tilde{d}, r) \cdot (a_{n_f},b_{n_f} , c_{n_f}).
    \]
    Then for $d\in D_r^*$, we have $\lim_{f\to \infty} (8r^2)!d \cdot (a_{n_f},b_{n_f}) \equiv a^* \pmod{1}$ for some $a^* \in \mathbb{R} / \mathbb{Z}$.
    We also have for $d\in D_r^*$, that $\lim_{f\to \infty}  ( \tilde{d}, r) \cdot (a_{n_f},b_{n_f} , c_{n_f}) \equiv b^* \pmod{1}$ for some $b^* \in \mathbb{R} / \mathbb{Z}$ with $b^*\not\equiv 0 \pmod{1}$, (because $(\tilde{d} , r) \cdot (a,b,c) \not\equiv 0 \pmod{1}$ for $d\in D_r^*$). Since $D_r^*$ is a finite set, by Lemma~\ref{geometric series}, it now follows that
    \[
    \lim_{f \to \infty} \sum_{d\in D_r^*} \sum_{\sigma= \pm 1} \sum_{j=1}^{n_f} \frac{n_f+1-j}{n_f(n_f+1)} e( \sigma ((8r^2)!d + j \tilde{d}, jr) \cdot (a_{n_f}, b_{n_f}, c_{n_f})) = 0.
    \]
    So, it remains to examine the summation over $D_r'=D_r\backslash D_r^*$.

    If $|D_r'|\le 2$, then the lemma now follows since $\sum_{j=1}^{n_f} \frac{n_f+1-j}{n_f(n_f+1)}= \frac{1}{2}$ for every positive integer $f$.
    So, we may now assume that $|D_r'|\ge 3$.
    
    Since $|D_r'|\ge 3$, there exists distinct $d_1,d_2,d_3\in D_r$ such that $\tilde{d_1} \cdot (a,b) \equiv \tilde{d_2} \cdot (a,b)  \equiv \tilde{d_3} \cdot (a,b)  \equiv  -rc \pmod{1}$. Then by Lemma~\ref{circle}, there exists integers $p_a,q_a,p_b,q_b$ such that $a=\frac{p_a}{q_a}$, $b=\frac{p_b}{q_b}$, and $1\le q_a,q_b \le 8r^2$.
    So, both $(8r^2)!/q_a$ and $(8r^2)!/q_b$ are integers.
    It follows that $(8r^2)!d\cdot (a,b)\equiv d\cdot (p_a(8r^2)/q_a,p_b(8r^2)/q_b)  \equiv 0 \pmod{1}$ for each $d\in D_r$.
    In particular, for each $d\in D_r$, we have that $\lim_{f\to \infty} (8r^2)!d\cdot (a_f,b_f)\equiv 0 \pmod{1}$.
    For each $d\in D_r'$, we also have that $\lim_{f\to \infty}  \tilde{d} \cdot ( \tilde{d}, r) \cdot (a_{n_f},b_{n_f} , c_{n_f}) \equiv 0 \pmod{1}$.
    Since $D_r'$ is a finite set, by Lemma~\ref{zero}, it now follows that
    \[
    \limsup_{f \to \infty} \sum_{d\in D_r'} \sum_{\sigma= \pm 1} \sum_{j=1}^{n_f} \frac{n_f+1-j}{n_f(n_f+1)} e( \sigma ((8r^2)!d + j \tilde{d}, jr) \cdot (a_{n_f}, b_{n_f}, c_{n_f})) \ge 0.
    \]
    Therefore, in the case that $|D_r'|\ge 3$, we have $\limsup_{f\to \infty} \widehat{w}_{r,n_f}(a_{n_f},b_{n_f}, c_{n_f}) \ge 0$, as desired.
\end{proof}

We are now ready to show that the Cayley graphs $G(\mathbb{Z}^3,C_r)$ have large chromatic number.

\begin{proposition}\label{prop}
    For each positive integer $r$, we have $\overline{\alpha}(G(\mathbb{Z}^3,C_r)) \le \frac{1}{1+|D_r|/2}$.
    As a consequence, $\chi(G(\mathbb{Z}^3,C_r)) \ge 1+|D_r|/2$.
\end{proposition}

\begin{proof}
    By Lemma~\ref{sup}, we have that $
    \sup_{u\in (\mathbb{R}/\mathbb{Z})^3} \widehat{w}_{r,n}(u) = |D_r|
    $ for every positive integer $n$.
    By Lemma~\ref{inf}, we have that
    $\limsup_{{n} \to \infty} \inf_{u\in (\mathbb{R} / \mathbb{Z})^3}  \widehat{w}_{r,n}(u) \ge -2$.
    Observe that for every positive integer $n$, we have that $\inf_{u\in (\mathbb{R} / \mathbb{Z})^3}  \widehat{w}_{r,n}(u)\le 0$, since 
    $\int_{[0,1]^3} \widehat{w}_{r,n}(u) \,du = 0$.
    For every positive integer $n$, $G(\mathbb{Z}^3, C_{r,n})$ is a subgraph of $G(\mathbb{Z}^3, C_{r})$.
    Therefore, by Theorem~\ref{ratio bound}, we have
    \begin{align*}
    \overline{\alpha}(G(\mathbb{Z}^3, C_{r} )  
    & \le \liminf_{{n} \to \infty} \overline{\alpha}(G(\mathbb{Z}^3, C_{r,n} )   \\
    & \le \liminf_{{n} \to \infty} 
	\frac{-\inf_{u\in (\mathbb{R} / \mathbb{Z})^3}  \widehat{w}_{r,n}(u) }{  \sup_{u\in (\mathbb{R} / \mathbb{Z})^3}  \widehat{w}_{r,n}(u)   -\inf_{u\in (\mathbb{R} / \mathbb{Z})^3}  \widehat{w}_{r,n}(u)},
    \\
    & =  
    \frac{-\limsup_{{n} \to \infty} \inf_{u\in (\mathbb{R} / \mathbb{Z})^3}  \widehat{w}_{r,n}(u) }{  |D_r|   -\limsup_{{n} \to \infty} \inf_{u\in (\mathbb{R} / \mathbb{Z})^3}  \widehat{w}_{r,n}(u)},
    \\
    & \le \frac{1}{1+|D_r|/2}.
    \end{align*}
    As a consequence, $\chi( G(\mathbb{Z}^3, C_{r} )  ) \ge 1/\overline{\alpha}(G(\mathbb{Z}^3, C_{r} ) \ge 1+|D_r|/2$.
\end{proof}

Theorem~\ref{main2} (and thus Theorem~\ref{main3}) now follows from Lemma~\ref{embed} and Proposition~\ref{prop}.

\section*{Acknowledgements}

The author thanks Jim Geelen for discussions on the problem of whether or not $\chi(\mathbb{R}^{2,2})$ is finite, which lead to this work.
The author also thanks Mohammad Bardestani, Keivan Mallahi-Karai, Vjekoslav Kovač, Mike Krebs, and the anonymous referee for a number of helpful comments that have improved the paper.
In particular, Mohammad Bardestani and Keivan Mallahi-Karai pointed out that Theorem \ref{main1} could be strengthened from a previous version of this paper (where ``$\mathbb{Q}^2$'' was replaced with ``$\mathbb{R}^2$''), answering a conjecture from that version.

\bibliographystyle{amsplain}

\end{document}